\documentclass[12pt]{conm-p-l}
\usepackage{amscd}
\usepackage{amsmath}
\usepackage{amsxtra}
\usepackage{amsfonts}
\usepackage{amssymb}
\usepackage{setspace}

\newtheorem{theorem}{Theorem}[section]
\newtheorem{corollary}[theorem]{Corollary}
\newtheorem{lemma}[theorem]{Lemma}
\newtheorem{proposition}[theorem]{Proposition}

\newtheorem{conjecture}[theorem]{Conjecture}
\theoremstyle{definition}
\newtheorem{definition}[theorem]{Definition}
\newtheorem{remark}[theorem]{Remark}

\newtheorem*{question}{Question}
\newtheorem{example}[theorem]{Example}
\theoremstyle{remark}

\renewcommand{\theclaim}{\textup{\theclaim}}

\numberwithin{equation}{section}

\def\openone

{\mathchoice

{\hbox{\upshape \small1\kern-3.3pt\normalsize1}}

{\hbox{\upshape \small1\kern-3.3pt\normalsize1}}

{\hbox{\upshape \tiny1\kern-2.3pt\SMALL1}}

{\hbox{\upshape \Tiny1\kern-2pt\tiny1}}}

\makeatletter

\newbox\ipbox

\newcommand{\diracb}[1]{\left\langle #1\mathrel{\mathchoice

{\setbox\ipbox=\hbox{$\displaystyle \left\langle\mathstrut
#1\right.$}

\vrule height\ht\ipbox width0.25pt depth\dp\ipbox}

{\setbox\ipbox=\hbox{$\textstyle \left\langle\mathstrut
#1\right.$}

\vrule height\ht\ipbox width0.25pt depth\dp\ipbox}

{\setbox\ipbox=\hbox{$\scriptstyle \left\langle\mathstrut
#1\right.$}

\vrule height\ht\ipbox width0.25pt depth\dp\ipbox}

{\setbox\ipbox=\hbox{$\scriptscriptstyle \left\langle\mathstrut
#1\right.$}

\vrule height\ht\ipbox width0.25pt depth\dp\ipbox}

}\right. }

\newcommand{\dirack}[1]{\left. \mathrel{\mathchoice

{\setbox\ipbox=\hbox{$\displaystyle \left.\mathstrut
#1\right\rangle$}

\vrule height\ht\ipbox width0.25pt depth\dp\ipbox}

{\setbox\ipbox=\hbox{$\textstyle \left.\mathstrut
#1\right\rangle$}

\vrule height\ht\ipbox width0.25pt depth\dp\ipbox}

{\setbox\ipbox=\hbox{$\scriptstyle \left.\mathstrut
#1\right\rangle$}

\vrule height\ht\ipbox width0.25pt depth\dp\ipbox}

{\setbox\ipbox=\hbox{$\scriptscriptstyle \left.\mathstrut
#1\right\rangle$}

\vrule height\ht\ipbox width0.25pt depth\dp\ipbox}

} #1\right\rangle}

\newcommand{\cj}[1]{\overline{#1}}

\newcommand{\bz}{\mathbb{Z}}

\newcommand{\br}{\mathbb{R}}
\newcommand{\bc}{\mathbb{C}}
\newcommand{\bt}{\mathbb{T}}
\newcommand{\bn}{\mathbb{N}}

\def\blfootnote{\xdef\@thefnmark{}\@footnotetext}

\renewcommand{\mod}{\operatorname{mod}}

\input xy
\xyoption{all}
\usepackage{amssymb}

\def\-{^{-1}}

\def\ty{\emptyset}

\def\lcm{\textup{lcm}}

\begin{document}
  \onehalfspacing

\title[On spectral sets of integers]{On spectral sets of integers}
\author{Dorin Ervin Dutkay}
\address{[Dorin Ervin Dutkay] University of Central Florida\\
	Department of Mathematics\\
	4000 Central Florida Blvd.\\
	P.O. Box 161364\\
	Orlando, FL 32816-1364\\
U.S.A.\\} \email{Dorin.Dutkay@ucf.edu}

\author{Isabelle Kraus}

\address{[Isabelle Kraus] University of Central Florida\\
	Department of Mathematics\\
	4393 Andromeda Loop N.\\
	Orlando, FL 32816-1364\\
U.S.A.\\} \email{izzy.kraus@knights.ucf.edu}

\thanks{} 

\begin{abstract}
Based on tiles and on the Coven-Meyerowitz property, we present some examples and some general constructions of spectral subsets of integers.
\end{abstract}
\maketitle \tableofcontents

\section{Introduction}

		\begin{definition}\label{def5.4}
	Let $G$ be a locally compact abelian group and let $\widehat G$ be its Pontryagin dual group. A finite subset $A$ of $G$ is called {\it spectral } (in $G$) if there exists a subset $\Lambda$ of $\widehat G$ with $\#\Lambda=\#A$ such that 
	\begin{equation}
	\frac{1}{\#A}\sum_{a\in A}\varphi(a)\cj\varphi'(a)=\delta_{\varphi\varphi'},\quad(\varphi,\varphi'\in \Lambda)
	\label{eq5.4.1}
	\end{equation}
In this case, $\Lambda$ is called a {\it spectrum} for $A$ (in the group $\widehat G$). 	
	\end{definition}
	
	It is easy to check that the spectral property can be rephrased in the following ways.
	
	\begin{proposition}\label{pr5.5}
	Let $A$ be a finite subset of a locally compact abelian group and $\Lambda$ a finite subset of $\widehat G$ with $\#\Lambda=\#A$. The following statements are equivalent:
	\begin{enumerate}
		\item $\Lambda$ is a spectrum for $A$. 
		\item The matrix 
		\begin{equation}
		\frac{1}{\sqrt{\#A}}\left(\varphi(a)\right)_{a\in A,\varphi\in\Lambda}
		\label{eq5.5.1}
		\end{equation}
		is unitary.
		\item For every $a,a'\in A$,
		\begin{equation}
		\frac{1}{\#\Lambda}\sum_{\varphi\in\Lambda}\varphi(a-a')=\delta_{aa'}.
		\label{eq5.5.2}
		\end{equation}
	\end{enumerate}
	\end{proposition}
	
	We will be interested mainly in the spectral subsets of $\bz$ and the spectral subsets of $\bz_N$. Since the dual group of $\bz$ is the group $\bt=\{z\in\bc :|z|=1\}$ which can be identified with $[0,1)$, a finite subset $A$ of $\bz$ is spectral if and only if there exists a finite subset $\Lambda$ in $[0,1)$ (or $\br$) such that $\#\Lambda=\#A$ and the matrix 
	$$\frac{1}{\sqrt{\#A}}\left(e^{2\pi i a\lambda}\right)_{a\in A,\lambda\in\Lambda}$$
	is unitary. 
	
	Since the dual group of $\bz_N$ is $\bz_N$, a subset $A$ of $\bz_N$ is spectral if and only if there exists a subset $\Lambda$ of $\bz_N$ such that  $\#A=\#\Lambda$ and the matrix 
	$$\frac1{\sqrt{\#A}}\left(e^{2\pi i a\lambda/N}\right)_{a\in A,\lambda\in \Lambda}$$
	is unitary.

	\begin{definition}\label{deftile}
	For two subsets $A$ and $B$ of $\bz$, we write $A\oplus B$ to indicate that for each $c\in A+B$ there are unique numbers $a\in A$ and $b\in B$ such that $a+b=c$. 
	
	For a set $A$ of non-negative integers we denote by $A(x)$ the {\it associated polynomial}
	$$A(x)=\sum_{a\in A}x^a.$$

	A subset $A$ of $\bz$ is called a {\it tile} if there exists a subset $C$ of $\bz$ such that $A\oplus C=\bz$. 
	\end{definition}

	In 1974 \cite{Fug74}, Fuglede proposed a conjecture that states that Lebesgue measurable spectral sets in $\br^n$ coincide with sets that tile $\br^n$. The conjecture was disproved by Tao \cite{Tao05} in dimensions five and higher and later in dimensions three and higher \cite{MR2267631,MR2237932,MR2264214,MR2159781}. All these counterexamples were based on some constructions in finite groups, so the Fuglede conjecture fails for groups of the form $\bz_{N_1}\times\bz_{N_2}\times\bz_{N_3}$.  However, the conjecture is still open at this moment in dimensions one and two. It is known that the Fuglede conjecture in $\br$, under some additional hypotheses, can be reduced to the Fuglede conjecture for $\bz$, see \cite{DuLa14}. 
	
	\begin{conjecture}\label{cofu}{\bf [Fuglede's conjecture for $\bz$]} A finite subset of $\bz$ is spectral if and only if it is a tile. 
	\end{conjecture}

	A basic result (see \cite{New77,CoMe99}) shows that every tiling set $C$ is {\it periodic}, i.e., there exists $N\in\bn$ such that $C+N=C$. If $B$ is any set consisting of one representative from $C$ for each class modulo $N$, then $C=B\oplus N\bz$ and so $A\oplus (B\oplus N\bz)=\bz$, and therefore $A\oplus B$ is a complete set of representatives modulo $N$. 
	
	\begin{proposition}\label{pr3.2}\cite{CoMe99}
Let $N$ be a positive integer and $A$, $B$ sets of non-negative integers. The following statements are equivalent:
\begin{enumerate}
	\item $A\oplus(B\oplus N\bz)=\bz$.
	\item $A\oplus B$ is a complete set of representatives of $\bz_N$. In other words $A\oplus B=\bz_N$, where addition is understood modulo $N$.
	\item $A(x)B(x)\equiv 1+x+\dots+x^{N-1}\mod (x^N-1)$. 
	\item $A(1)B(1)=N$ and for every factor $t>1$ of $N$, the cyclotomic polynomial $\Phi_t(x)$ divides $A(x)$ or $B(x)$. 
\end{enumerate}

\end{proposition}

Thus, tiles for $\bz$ coincide with tiles for the groups $\bz_N$. 

In \cite{CoMe99}, Coven and Meyerowitz found a sufficient condition for a subset of $\bz$ to be a tile, formulated in terms of cyclotomic polynomials. 
	
\begin{theorem}\label{thcm}\cite{CoMe99} Let $A$ be a finite set of non-negative integers with corresponding polynomial $A(x)=\sum_{a\in A}x^a$. Let $S_A$ be the set of prime powers $s$ such that the cyclotomic polynomial $\Phi_s(x)$ divides $A(x)$. Consider the following conditions on $A(x)$. 
\begin{enumerate}
	\item[(T1)] $A(1)=\prod_{s\in S_A}\Phi_s(1)$. 
	\item[(T2)] If $s_1,\dots,s_m\in S_A$ are powers of distinct primes, then the cyclotomic polynomial $\Phi_{s_1\dots s_m}(x)$ divides $A(x)$.
	
	\end{enumerate}
	If $A(x)$ satisfies (T1) and (T2), then $A$ tiles the integers with period $N:=\lcm(S_A)$. The tiling set $B$ can be obtained as follows: define $B(x)=\prod \Phi_s(x^{t(s)})$, where the product is taken over all prime power factors $s$ of $N$ which are not in $S_A$, and $t(s)$ is the largest factor of $N$ relatively prime to $s$. Then $B(x)$ is the polynomial associated to a set of non-negative integers $B$. 
	\end{theorem}
	
		\begin{definition}\label{defcm}
	A finite set $A$ of non-negative integers is said to satisfy the {\it Coven-Meyerowitz property} (or the CM-property), if  it satisfies conditions (T1) and (T2) in Theorem \ref{thcm}. We call the tiling set $B$ in Theorem \ref{thcm}, the {\it Coven-Meyerowitz (CM) tiling set} associated to $A$, and we denote it by $B=\operatorname{CM}(A)$.
	\end{definition}
	
	The converse of the Coven-Meyerowitz theorem also seems to be true, but at the moment, it is just a conjecture. Coven and Meyerowitz showed that tiles satisfy the (T1) property.
		\begin{theorem}\label{th4.13}\cite{CoMe99} Let $A$ be a finite set of non-negative integers with corresponding polynomial $A(x)=\sum_{a\in A}x^a$ and let $S_A$ be the set of prime powers $s$ such that the cyclotomic polynomial $\Phi_s(x)$ divides $A(x)$. If $A$ tiles the integers, then (T1) holds. 
	\end{theorem}

Also, they proved that tiles with  a cardinality that has only one or two prime factors satisfy the CM-property. 

\begin{theorem}\label{thcob2}\cite{CoMe99}
Let $A$ be a finite set of non-negative integers with corresponding polynomial $A(x)=\sum_{a\in A}x^a$ such that $\#A$ has at most two prime factors and $A$ tiles $\bz$. Then $A$ satisfies (T2), and therefore it has the CM-property. 
\end{theorem}

	\begin{remark}\label{remcm}
	Note that the CM-tiling set $B$ or $C_S=B\oplus\lcm(S_A)\bz$ does not depend on $A$, it depends only on $S=S_A$. Also, the proof shows that if $A$ satisfies (T1) and (T2), then it has a universal tiling of period $\lcm(S_A)$, which is a tiling set for all the sets $A'$ with $S_{A'}=S_A$. 
	
	Note also that $B,C_S\subset p\bz$ for every prime factor $p\in S_A$ since $p$ is a factor of $N$ and every divisor $\Phi_s(x^{t(s)})$ of $B(x)$ is a polynomial in $x^p$. This is because either $s=p^{\alpha+1}$ with $\alpha\geq 1$, and then we use Proposition \ref{pr2.15}(iii), or $s$ is a power of a prime different than $p$ and then $t(s)$ is a multiple of $p$, being the largest factor of $\lcm(S_A)$ relatively prime to $s$.  
	\end{remark}
	
	Later, \L aba proved that sets with the CM-property are spectral. 
		\begin{theorem}\label{thlab}\cite{Lab02}
	Let $A$ be set of non-negative integers that satisfies the CM-property. Then $A$ is a spectral set. A spectrum for $A$ can be obtained as follows: consider the set $\Lambda_A$ of all numbers of the form 
	$$\sum_{s\in S_A}\frac{k_s}{s},$$
	where $k_s\in\{0,1,\dots,p-1\}$ if $s=p^{\alpha}\in S_A$, with $p$ prime. 
	
	\end{theorem}

	\begin{definition}\label{deflabasp}
	With the notations as in Theorems \ref{thcm} and \ref{thlab}, for $N:=\lcm(S_A)$, we denote by $\textup{\L}_A:=\Lambda_A$ and we call $\textup{\L}_A$ the {\it \L aba spectrum} of $A$. 
	\end{definition}
	
		Combining the results of Coven-Meyerowitz and \L aba, Dutkay and Haussermann showed that if a set has the CM-property, then the tiling sets and the spectra are in a nice complementary relation. 

	\begin{theorem}\label{thdh}\cite{DH15}
	Let $A$ be a finite set of non-negative integers with the CM-property. Let $N=\lcm(S_A)$, and let $B=\textup{CM}(A)$ be its CM-tiling set. Then $B$ has the CM-property, and if $\textup{\L}_A$, $\textup{\L}_B$ are the corresponding \L aba spectra, then $(N\cdot\textup{\L}_A)\oplus(N\cdot\textup{\L}_B)=\bz_N$. 
	\end{theorem}
	
	Many examples of tiles are found in the literature. Many fewer spectral sets are known. In this paper, we gather some of the examples of tiling sets in the literature and show that they have the CM-property and explicitly describe the tiling sets and the spectra. In Section 2, we describe the tiles with cardinality of a prime power. In Section 3, we describe Szab{\'o}'s examples and show that they have the CM-property and describe the tiling sets and spectra. In Section 4, we present some general constructions of spectral sets, tiling sets, and sets with the CM-property.
	
	\section{One prime power}

	\begin{theorem}\label{th2.1}
	Let $A$ be a set of non-negative integers with cardinality $p^n$, where $p$ is prime and $n\in\bn$. Then the set $A$ tiles the integers if and only if there exist integers $1\leq\alpha_1<\dots<\alpha_n$ and for each $0\leq k\leq n$, and each $1\leq i_1,\dots,i_{k-1}\leq p-1$ there exists a complete set of representatives modulo $p$, $\{a_{i_1,\dots,i_{k-1},i_k} :0\leq i_k\leq p-1\}$, $a_{i_1,\dots,i_{k-1},0}=0$, such that the set $A$ is congruent modulo $p^{\alpha_n}$ to the set 
	
	\begin{equation}
	A'=\left\{\sum_{k=1}^n p^{\alpha_k-1}a_{i_1,\dots,i_k} : 0\leq i_1,\dots,i_n\leq n-1\right\}.
	\label{eq2.1.1}
	\end{equation}
	
	In this case $S_A=\{p^{\alpha_1},\dots,p^{\alpha_n}\}$, the CM-tiling set is 
	\begin{equation}
	B=\textup{CM}(A)=\left\{\sum_{\overset{j=0,\dots,\alpha_n-1}{j\neq \alpha_1-1,\dots,\alpha_n-1}}b_jp^j :0\leq b_j\leq p-1, 0\leq j\leq \alpha_n-1,j\neq \alpha_1-1,\dots,\alpha_n-1\right\}.
	\label{eq2.1.2}
	\end{equation}
	The \L aba spectra of $A$ and $B$ are, respectively
	\begin{equation}
	\textup{\L}_A=\left\{ \sum_{i=1}^n\frac{k_i}{p^{\alpha_i}} : 0\leq k_i\leq p-1, 1\leq i\leq n
	\right\}
	\label{eq2.1.3}
	\end{equation}
	\begin{equation}
	\textup{\L}_B=\left\{\sum_{\overset{j=1,\dots,\alpha_n}{j\neq \alpha_1,\dots,\alpha_n}} \frac{k_j}{p^j} : 0\leq k_j\leq p-1, 1\leq j\leq \alpha_n, j\neq \alpha_1,\dots,\alpha_n\right\}
	\label{eq2.1.4}
	\end{equation}
	\end{theorem}

	\begin{remark}
	Let us explain a bit more the structure of the set $A'$. Think of the base $p$ decomposition of a number. For the set $A'$, we only use the digits corresponding to positions $\alpha_1-1, \alpha_2-1, \dots, \alpha_n -1$.  The rest of the digits are 0. In position $\alpha_1 -1$ we use a complete set of representatives modulo $p$, $\{a_{i_1} : 0\leq i_1\leq p-1\}$ with $a_0=0$. Once the first digit $a_{i_1}$ is chosen for the digit in position $\alpha_1-1$, we use another complete set of representatives modulo $p$, $\{a_{i_1,i_2}: 0\leq i_2\leq p-1\}$, with $a_{i_1,0}=0$. Note that, this complete set of representatives is allowed to be different for different choices of $i_1$. 
	
	For $1\leq k\leq n$, once the digits $a_{i_1},a_{i_1,i_2}, \dots, a_{i_1,\dots,i_{k-1}}$ have been chosen for positions $\alpha_1-1,\alpha_2-1,\dots,\alpha_{k-1}-1$ respectively, for the digit in position $\alpha_k-1$ we pick a complete set of representatives, $\{a_{i_1,\dots,i_k} :0\leq i_k\leq p-1\}$, with $a_{i_1,\dots,i_{k-1},0}=0$. 
	
	\end{remark}
	
	We will need some results from \cite{CoMe99}.
	
	\begin{definition}\label{def4.2}
	Let $S$ be a set of powers of at most two primes. Define $\mathcal T_S$ to be the collection of all subsets $A$ of $\{0,1,\dots, \lcm(S)-1\}$ which tile the integers and satisfy $\min(A)=0$ and $S_A=S$. Note that $\mathcal T_\ty=\{0\}$ because $\lcm(\ty)=1$. 	
	\end{definition}
	
	\begin{lemma}\label{lem4.3}
	Let $S$ be a set of powers of at most two primes. 
	A finite set $A'$ with $\min(A')=0$ and $S_{A'}=S$ tiles the integers if and only if $A'$ is congruent modulo $\lcm(S)$ to a member of $\mathcal T_S$. 
	\end{lemma}
	
	\begin{proof}
	Let $A$ be an element of $\mathcal T_S$ and $A'\equiv A(\mod \lcm(S))$. Let $N=\lcm(S)$. Since $A'\in\mathcal T_S$, there exists a set $B$ such that $A'\oplus B=\bz_N$. Then, since $A\equiv A' (\mod N)$, it follows that $A\oplus B=\bz_N$. 
	
	Conversely, if $A'$ tiles the integers and $S_A=S$, then by Lemma \ref{lem4.1}, $\#A'$ has at most two prime factors, so it has the CM-property, by Theorem \ref{thcob2}. Therefore, it has a tiling set of period $\lcm(S)$, $B\oplus \lcm(S)\bz$, by Remark \ref{remcm}. Let $A$ be the set obtained from $A'$ by reducing modulo $\lcm(S)$. Then $\min(A)=0$ and $A\subset\{0,1,\dots,\lcm(S)-1\}$. Also, $A$ has the same tiling set $B\oplus \lcm(S)\bz$. Then, by Lemma \ref{lemco2.1}, $S_{A'}=S_A$ as the complement of $S_B$ in the set of all prime power factors of $\lcm(S)$. 
	\end{proof}
		
\begin{lemma}\label{lem4.1}
	Let $A$ be a finite set of non-negative integers which is a tile. Then $\#A$ has at most two prime factors if and only if $S_A$ consists of powers of at most two primes. 
	\end{lemma}
	
	\begin{proof}
	Since $A$ is a tile, it satisfies the (T1) property, by Theorem \ref{th4.13}. So, using Proposition \ref{pr2.15}(iv),
	$$\#A=\prod_{s\in S_A}\Phi_s(1)=\prod_{p^\alpha\in S_A}p.$$
	Thus $\#A$ has at most two prime factors if and only if $S_A$ consists of powers of at most two primes. 
	\end{proof}

		\begin{lemma}\cite{CoMe99}\label{lem4.4}
	Suppose $S$ contains powers of only one prime $p$. Let $\cj S=\{p^\alpha : p^{\alpha+1}\in S, \alpha\geq 1\}$. 
	\begin{enumerate}
		\item If $p\not\in S$ then $$\mathcal T_S=\{p\cj A : \cj A\in \mathcal T_{\cj S}\}.$$
		\item If $p\in S$ then 
		$$\mathcal T_{S}=\left\{\cup_{i=0}^{p-1}(\{a_i\}\oplus p\cj A_i) : \cj A_i\in \mathcal T_{\cj S}, a_0=0,\{a_0,a_1,\dots, a_{p-1}\}\mbox{ a complete }\right.$$
		$$\left.\mbox{set of representatives modulo $p$ and every }\{a_i\}\oplus p \cj A_i\subset \{0,1,\dots,\lcm(S)-1\}\right\}.$$
	
	\end{enumerate}
	\end{lemma}
	\begin{proposition}\cite{CoMe99}\label{prex1}
	Let $p$ be a prime number. Then
	\begin{enumerate}
		\item The only member of $\mathcal T_\ty$ is $\{0\}$. 
		\item For $\alpha\geq 0$, the only member of $\mathcal T_{\{p^{\alpha+1}\}}$ is $p^\alpha\{0,1,\dots, p-1\}$. 
	\end{enumerate}
	\end{proposition}

	\begin{theorem}\label{th5.14}
	Let $p$ be a prime number. Let $S=\{p^{\alpha_1},p^{\alpha_2},\dots,p^{\alpha_n}\}$ with $1\leq\alpha_1<\alpha_2<\dots<\alpha_n\}$. The following statements are equivalent:
	\begin{enumerate}
		\item $A\in\mathcal T_S$.
		\item For $1\leq k\leq n-1$, there exist numbers $a_{i_1,\dots,i_k}$ $i_1,\dots,i_k=0,1\dots,p-1$ with the following properties
		\begin{enumerate}
			\item The set $\{a_{i_1} :0\leq i_1\leq p-1\}$ is a complete set of representatives modulo $p$, $a_0=0$,
			\item For each $2\leq k\leq n-1$, and each $i_1,\dots,i_{k-1}$ in $\{0,1\dots,p-1\}$, the set $\{a_{i_1,\dots,i_{k-1},i_k} :0\leq i_k\leq p-1\}$ is a complete set of representatives modulo $p$, $a_{i_1,\dots,i_{k-1},0}=0$, 
			\item For each $1\leq k\leq n-1$
			\begin{equation}
			a_{i_1,\dots,i_k}+p^{\alpha_{k+1}-\alpha_k}a_{i_1,\dots,i_{k+1}}+\dots+p^{\alpha_{n-1}-\alpha_k}a_{i_1,\dots,i_{n-1}}\leq p^{\alpha_n-\alpha_k}-1.
			\label{eq5.14.1}
			\end{equation}
			\item 
			\begin{equation}
			A=\left\{p^{\alpha_1-1}a_{i_1}+p^{\alpha_2-1}a_{i_1,i_2}+\dots+p^{\alpha_{n-1}-1}a_{i_1,\dots,i_{n-1}}+p^{\alpha_n-1} j :0\leq i_1,\dots, i_n,j\leq p-1\right\}.
			\label{eq5.14.2}
			\end{equation}
		\end{enumerate}
	
	\end{enumerate}
	\end{theorem}
	
	\begin{proof}
	We prove the equivalence of (i) and (ii) by induction on $n$. For $n=1$, the result follows from Proposition \ref{prex1}. Assume now the statements are equivalent for $n$ and take $S=\{p^{\alpha_1},\dots,p^{\alpha_{n+1}}\}$. Using Lemma \ref{lem4.4}(i), we have that $A\in \mathcal T_S$ if and only if $A=p^{\alpha_1-1}\cj A$ with $\cj A\in \mathcal T_{\cj S_1}$ where $\cj S=\{p,p^{\alpha_2-\alpha_1+1},\dots,p^{\alpha_{n+1}-\alpha_1+1}\}$. Using Lemma \ref{lem4.4}(ii), we have that $\cj A\in\mathcal T_{\cj S}$ if and only if there exists a complete set of representatives modulo $p$, $\{a_{i_1} : 0\leq i_1\leq p-1\}$, $a_0=0$ and, for each $0\leq i_1\leq p-1$, a set $\cj A_{i_1}\in \mathcal T_{\cj S'}$, where $\cj S'=\{p^{\alpha_2-\alpha_1},\dots, p^{\alpha_{n+1}-\alpha_1}\}$ such that $a_{i_1}+p\cj A_{i_1}\subset\{0,\dots,\lcm(\cj S)-1\}=\{0,\dots,p^{\alpha_{n+1}-\alpha_1+1}-1\}$, and 
	$$\cj A=\cup_{i_1=0}^{p-1}(\{a_{i_1}\}+p\cj A_{i_1}).$$
	
	Using the induction hypothesis for the set $\cj S'$, we get that, for each $0\leq i_1\leq p-1$ the set $\cj A_{i_1}$ must be of the form 
	$$\cj A_{i_1}=\left\{p^{\alpha_2-\alpha_1-1}a_{i_1,i_2}+p^{\alpha_3-\alpha_1-1}a_{i_1,i_2,i_3}+\dots+p^{\alpha_{n}-\alpha_1-1}a_{i_1,\dots,i_{n}}\right.$$$$\left.+p^{\alpha_{n+1}-\alpha_1-1}j :0\leq i_2,\dots,i_{n},j\leq p-1\right\},$$
	where for each $0\leq i_2,\dots, i_{k-1}\leq p-1$, the set $\{a_{i_1,\dots, i_k} : 0\leq i_k\leq p-1\}$ is a complete set of representatives modulo $p$ and $a_{i_1,\dots,i_{k-1},0}=0$. 
	Also 
	$$a_{i_1,i_2,\dots,i_k}+p^{(\alpha_{k+1}-\alpha_1)-(\alpha_k-\alpha_1)}a_{i_1,\dots,i_{k+1}}+\dots+p^{(\alpha_n-\alpha_1)-(\alpha_k-\alpha_1)}\leq p^{(\alpha_{n+1}-\alpha_1)-(\alpha_k-\alpha_1)}-1$$
	and this implies (c) for $k\geq 2$. 
	We must also have 
	$$a_{i_1}+p^{\alpha_2-\alpha_1}a_{i_1,i_2}+\dots p^{\alpha_n-\alpha_1}a_{i_1,\dots,i_n}+p^{\alpha_{n+1}-\alpha_1}(p-1)\leq p^{\alpha_{n+1}-\alpha_1+1},$$
	and this implies (c) for $k=1$. 
	
	Then 
	$$A=p^{\alpha_1-1}\cup_{i_1=0}^{p-1}(\{a_{i_1}\}+p\cj A_{i_1}),$$
	and (d) follows. 
	
	\end{proof}
	\begin{proof}[Proof of Theorem \ref{th2.1}]
	Assume that $A$ and $A'$ have the given form. We show that $A'\oplus B=\bz_{p^{\alpha_n}}$. Note that $\#A'=p^n$ and $\#B=p^{\alpha_n-n}$ so $\#A'\cdot\#B=p^{\alpha_n}$. By Lemma \ref{lema1}, it is enough to show that $(A'-A')\cap (B-B)=\{0\}$ in $\bz_{p^{\alpha_n}}$. If we pick an element in the intersection, it can be written in both ways as 
	$$\sum_{k=1}^n p^{\alpha_k-1}(a_{i_1,\dots,i_k}-a_{i_1',\dots,i_k'})=\sum_{\overset{j=0,\dots,\alpha_n-1}{j\neq \alpha_1-1,\dots,\alpha_n-1}}p^j(b_j-b_j').$$ 
	
	Take the first index $l$ such that $i_l\neq i_l'$. Then the left-hand side is divisible by $p^{\alpha_l-1}$ but not by $p^{\alpha_l}$. Since $p^{\alpha_l-1}$ does not appear on the right hand side, it follows, by contradiction, that both sides are equal to 0. 
	
	For the converse, if $A$ is a tile then the result follows from Theorem \ref{th5.14}. 
	
	It remains to check that the CM-tiling set and the \L aba spectra are those given in \eqref{eq2.1.2},\eqref{eq2.1.3} and \eqref{eq2.1.4}. By Lemma \ref{lemco2.1}, we have that $S_A$ and $S_B$ are complementary, so 
	$$S_B=\left\{p^j : j\in\{1,\dots,\alpha_n\}, j\neq\alpha_1,\dots,\alpha_n\right\}.$$
	Then the CM-tiling set is defined by the polynomial 
	$$\prod_{\overset{j=1\dots\alpha_n}{j\neq\alpha_1,\dots,\alpha_n}}\Phi_j(x)=\prod_{\overset{j=1\dots\alpha_n}{j\neq\alpha_1,\dots,\alpha_n}}(1+x^{p^{j-1}}+x^{2p^{j-1}}+\dots+x^{(p-1)p^{j-1}})$$
	$$=\sum_{\overset{0\leq b_j\leq p-1}{1\leq j\leq \alpha_n, j\neq \alpha_1,\dots,\alpha_n}}x^{\sum_{1\leq j\leq \alpha_n, j\neq \alpha_1,\dots,\alpha_n}b_jx^{p^{j-1}}}.$$
	This implies \eqref{eq2.1.2}. Since we have the form of $S_B$, \eqref{eq2.1.3} and \eqref{eq2.1.4} follow immediately.

	\end{proof}
	
	
	\begin{remark}
	In \cite{New77}, Newman classifies the finite sets of integers which tile $\mathbb{Z}$ when the number of elements in the set is a prime power. The tiling condition is stated in Theorem \ref{thnew}. In Proposition \ref{thnewnew}, we determine the relation between the numbers $e_{ij}$ in Newman's paper and the numbers $\alpha_i$ and the set $S_A$ in Theorem \ref{th2.1}.
	\end{remark}
	
	\begin{theorem}\label{thnew}\cite{New77}
	Let $a_1, a_2, \ldots, a_k$ be distinct integers with $k = p^{\alpha}$, $p$ a prime, $\alpha$ a positive integer. For each pair $a_i, a_j, i \neq j$, we denote by $p^{e_{ij}}$ the highest power of $p$ which divides $a_i - a_j$. The set $\{a_1, a_2, \ldots, a_k\}$ is a tile if and only if there are at most $\alpha$ distinct $e_{ij}$. 
	\end{theorem}
	
	\begin{proposition}\label{thnewnew}
	Let $A = \{a_1, \ldots, a_k \}$ be a set of non-negative integers which tile $\bz$, with $\#A = p^{n}$, $p$ a prime, $n$ a positive integer. Then $S_A = \{p^{e_{ij}+1} : 1 \leq i, j \leq p^{n}\}$, where $e_{ij}$ denotes the highest power of $p$ which divides $a_i - a_j$, as in Theorem \ref{thnew}.
	\end{proposition}
	
	\begin{proof}
	Let $S_A = \{p^{\alpha_1}, \ldots, p^{\alpha_n}\}$. Since $A$ is a tile, we have by Theorem \ref{th2.1} that $A$ is congruent modulo $p^{\alpha_n}$ to the set 
$$A'=\left\{\sum_{k=1}^n p^{\alpha_k-1}a_{i_1,\dots,i_k} : 0\leq i_1,\dots,i_n\leq n-1\right\}.$$
Take $a, a' \in A$. Then
$$a={\sum_{k=1}^n p^{\alpha_k-1}a_{i_1,\dots,i_k} + p^{\alpha_n}m},$$
$$a'={\sum_{k=1}^n p^{\alpha_k-1}a_{i_{1'},\dots,i_{k'}} + p^{\alpha_n}m'}.$$
So
$$a-a'={\sum_{k=1}^n (a_{i_1,\dots,i_k}-a_{i_{1'},\dots,i_{k'}})p^{\alpha_k-1} + p^{\alpha_n}(m-m')}.$$
If $i_1 = i_{1'}, \ldots, i_n = i_{n'}$, then $a = a'$. So then $m = m'$.
 If there exists a $k$ such that $i_k \neq i_{k'}$, then take the smallest such $k$. Let $e$ be the largest power of $p$ such that $p^e$ divides $a-a'$. Then
$$a-a'=p^{\alpha_k-1}(a_{{i_1}, \ldots, i_{k-1}, i_k} - a_{{i_{1'}}, \ldots, i_{{k-1}'}, i_{k'}}) + {\sum_{j=k+1}^n (a_{i_1,\dots,i_j}-a_{i_{1'},\dots,i_{j'}})p^{\alpha_j-1} + p^{\alpha_n}(m-m')}.$$
Since $\{a_{i_1, \ldots , i_{k-1},l} : 0 \leq l \leq p-1\}$ is a complete set of representatives modulo $p$, this means that $a-a'$ is divisible by $p^{\alpha_k-1}$, but not $p^{\alpha_k}$. Therefore, $e = \alpha_k - 1$. Relabeling $a = a_i$ and $a' = a_j$, we get that $e_{ij} = \alpha_k-1$. Doing this for all $a_i, a_j \in A$, we get that the set of all $e_{ij}$ is $\{\alpha_1 - 1, \ldots, \alpha_n - 1\}$ and the result follows.
	\end{proof}
	

	For the case when the cardinality of a tile has two prime factors, to describe the structure of the such tiles, one can use the following lemma from \cite{CoMe99}.

	\begin{lemma}\label{lem4.5}
	Suppose $S$ contains powers of both the primes $p$ and $q$. Let 
	$$\cj S=\{p^\alpha : p^{\alpha+1}\in S\}\cup \{q^\beta : q^\beta\in S\},\quad \cj S'=\{p^\alpha : p^\alpha\in S\}\cup \{q^\beta : q^{\beta+1}\in S\}.$$
	\begin{enumerate}
		\item If $p\in S$, then 
		$$\mathcal T_{S}=\left\{\cup_{i=0}^{p-1}(\{a_i\}\oplus p\cj A_i) : \cj A_i\in \mathcal T_{\cj S}, a_0=0,\{a_0,a_1,\dots, a_{p-1}\}\mbox{ a complete }\right.$$
		$$\left.\mbox{set of representatives modulo $p$ and every }\{a_i\}\oplus p \cj A_i\subset \{0,1,\dots,\lcm(S)-1\}\right\}.$$
		\item If $q\in S$, then
		$$\mathcal T_{S}=\left\{\cup_{i=0}^{q-1}(\{a_i\}\oplus q\cj A_i) : \cj A_i\in \mathcal T_{\cj S'}, a_0=0,\{a_0,a_1,\dots, a_{q-1}\}\mbox{ a complete }\right.$$
		$$\left.\mbox{set of representatives modulo $q$ and every }\{a_i\}\oplus q \cj A_i\subset \{0,1,\dots,\lcm(S)-1\}\right\}.$$
		\item If $p,q\not\in S$, then $A\subset p\bz$ or $A\subset q\bz$ and 
		$$\{A\in \mathcal T_S : A\subset p\bz\}=\{p\cj A : \cj A\in\mathcal T_{\cj S}\},\quad 
		\{A\in \mathcal T_S : A\subset q\bz\}=\{q\cj A : \cj A\in\mathcal T_{\cj S'}\}. $$
		
	\end{enumerate}
	\end{lemma}

	\section{Szab{\'o}'s examples }
	In \cite{Sza85}, Szab{\'o} constructed a class of examples to give a negative answer to two questions due to K. Corr{\'a}di and A.D. Sands respectively:
	
	If $G$ is a finite abelian group and $G=A_1\oplus A_2$ is one of its normed factorizations, in the sense that both factors contain the zero, must one of the factors contain some proper subgroup of $G$?
	
	If $G$ is a finite abelian group and $G=A_1\oplus A_2$ is a factorization, must one of the factors be contained in some proper subgroup of $G$?
	
	We will present here Szab{\'o'}s examples and prove that they have the CM-property and find their tiling sets and spectra. We begin with a general proposition which encapsulates the core ideas in Szab{\'o}'s examples. 
	
	\begin{proposition}\label{pr5.7} Let $G$ be a finite abelian group.
	Suppose we have a factorization $A\oplus B'=G$. Let $B_1,\dots,B_r$ be disjoint subsets of $B'$ and $g_1,\dots,g_r$ in $G$ with the property that 
	\begin{equation}
	A+B_i=A+B_i+g_i\mbox{ for all }i.
	\label{eq5.7.1}
	\end{equation}
	Define 
	\begin{equation}
	B=B'\cup\left(\bigcup_{i=1}^r(B_i+g_i)\right)\setminus\left(\bigcup_{i=1}^rB_i\right).
	\label{eq5.7.2}
	\end{equation}
	
	Then $A\oplus B=G$. 
	\end{proposition}
	
	\begin{proof}
	The sets $A+B_i$ are disjoint. Indeed, if not, then there are $a,a'\in A$, $b_i\in B_i$, $b_j\in B_j$, $i\neq j$ such that $a+b_i=a'+b_j$. Because of the unique factorization, $a=a'$ and $b_i=b_j$, which contradicts the fact that $B_i$ and $B_j$ are disjoint. 
	We have the disjoint union 
	$$G=(A+(B'\setminus\cup_i B_i))\cup \bigcup_i(A+B_i)=(A+(B'\setminus \cup_i B_i))\cup\bigcup_i(A+B_i+g_i).$$
	Take $b\neq b'$ in $B$ and we want to prove that $A+b$ is disjoint from $A+b'$. If $b$ or $b'$ is in $B'\setminus\cup_i B_i$, this is clear from the hypothesis. 
	If $b$ and $b'$ lie in different sets $B_i+g_i$ this is again clear since the union above is disjoint. If $b,b'\in B_i+g_i$, then $b=c+g_i$, $b'=c'+g_i$ with $c,c'\in B_i\subset B'$, and since $b\neq b'$, we have $c\neq c'$. Therefore $A+c$ is disjoint from $A+c'$ and the same is true for $A+b=A+c+g_i$ and $A+b'=A+c'+g_i$. Thus $A\oplus B=G$.
	\end{proof}

	Szab{\'o} constructed his examples in $\bz_{m_1}\times\dots\times\bz_{m_r}$. We note here that working in the group $\bz_{m_1}\times\dots\bz_{m_r}$ with $m_1,\dots,m_r$ relatively prime is equivalent to working in $\bz_{m_1\dots m_r}$ because of the following isomorphism.

	\begin{proposition}\label{pr5.6.1}
	Let $m_1,\dots,m_r$ be relatively prime non-negative integers. Let $m=m_1\dots m_r$. 
	The map $\Psi:\bz_{m_1}\times\dots\times\bz_{m_r}\rightarrow \bz_m$,
	\begin{equation}
	\Psi(k_1,\dots,k_r)=\sum_{i=1}^rk_i\frac m{m_i}
	\label{eq5.6.1.1}
	\end{equation}
	is an isomorphism. 
	\end{proposition}
	
	\begin{proof}
	It is clear that $\Psi$ is a morphism. We check that it is injective. If $\Psi(k_1,\dots,k_r)=0$ then $\sum_{i=1}^rk_i\frac{m}{m_i}=0$. Since $m/m_i$ is divisible by $m_j$ for all $i\neq j$, then $k_j\frac{m}{m_j}\equiv 0(\mod m_j)$. Since $\frac{m}{m_j}$ is relatively prime to $m_j$, it is invertible in $\bz_{m_j}$ so $k_j\equiv 0(\mod m_j)$. Thus $(k_1,\dots,k_r)=(0,0,\dots,0)$. Thus, $\Psi$ is injective. Since the two sets have the same cardinality, $\Psi$ is also surjective.
	\end{proof}

	\begin{example}\label{exsz}\cite{Sza85}
	Let $G$ be the direct product of the cyclic groups of orders $m_1,\dots, m_r$ and generators $g_1,\dots, g_r$ respectively. Assume $r\geq 3$. We can think of $G$ as $\bz_{m_1}\times\dots\times\bz_{m_r}$ and we can pick the generators $g_1=(1,0,\dots,0), \dots, g_r=(0,\dots,0,1)$. But we can also pick other generators, for example $g_1=(q,0\dots,0)$ where $q$ is some number which is relatively prime to $m_1$. Let $\pi$ be a permutation of the set $\{1,\dots,r\}$  that does not have cycles of length 1 or 2. (Szab{\'o} assumes $\pi$ is cyclic, and since $r\geq 3$, our condition is satisfied, but we do not need $\pi$ to be cyclic.) 
	
	If $g\in G\setminus\{0\}$ and $m$ is a positive integer which is less than or equal to the order of $g$, we denote by $[g]_m=\{0,g,2g,\dots,(m-1)g\}$. 
	
	Assume now $m_i=u_iv_i$ where $u_i,v_i$ are integers greater than one. 
	
	Obviously 
	$$G=\sum_{i=1}^r[g_i]_{m_i}=\sum_{i=1}^r([g_i]_{u_i}+[u_ig_i]_{v_i})=A\oplus B',$$
	where 
	$$A=\sum_{i=1}^r[g_i]_{u_i}\mbox{ and } B'=\sum_{i=1}^r[u_ig_i]_{v_i}.$$
	
	Then Szab{\'o} picks $B_i=[u_ig_i]_{v_i}+u_{\pi(i)}g_{\pi(i)}$ and $g_i=g_i$, as in Proposition \ref{pr5.7}. An easy check shows that 
	the sets $B_i$ are disjoint (here is where we need $\pi$ to have only cycles of length $\geq 3$) and $A+B_i+g_i=A+B_i$ (the main property used here is that $[g_i]_{u_i}+[u_ig_i]_{v_i}+g_i=[g_i]_{m_i}+g_i=[g_i]_{m_i}=[g_i]_{u_i}+[u_ig_i]_{v_i}.$ Thus, the properties in Proposition \ref{pr5.7} are satisfied and with 
	$$B=B'\cup\left(\bigcup_{i=1}^r([u_ig_i]_{v_i}+u_{\pi(i)}g_{\pi(i)}+g_i)\right)\setminus\left(\bigcup_{i=1}^r([u_ig_i]_{v_i}+u_{\pi (i)}g_{\pi(i)})\right),$$
	we have a new factorization $A\oplus B=G$. 
	
	Next, we construct spectra for the sets $A,B'$ and $B$. For this, we regard $G$ as we mentioned before: $G=\bz_{m_1}\times\dots\times\bz_{m_r}$ and  $g_1=(1,0,\dots,0), \dots, g_r=(0,\dots,0,1)$.
	 
	Then we have 
	$$A=\sum_{i=1}^r[g_i]_{u_i}=\left\{(k_1,\dots,k_r) : 0\leq k_i\leq u_i-1\mbox{ for all }i\right\},$$
	$$B'=\sum_{i=1}^r[u_ig_i]_{v_i}=\left\{((u_1k_1,\dots,u_rk_r): 0\leq k_i\leq v_i-1\right\},$$
	$$B=B'\cup\left(\bigcup_{i=1}^r([u_ig_i]_{v_i}+u_{\pi(i)}g_{\pi(i)}+g_i)\right)\setminus\left(\bigcup_{i=1}^r([u_ig_i]_{v_i}+u_{\pi (i)}g_{\pi(i)})\right)$$
	$$=\left\{((u_1k_1,\dots,u_rk_r): 0\leq k_i\leq v_i-1\right\}\cup$$$$\bigcup_{i=1}^r\left\{(0,\dots,0,k_iu_i+1,0,\dots,0,u_{\pi(i)},0,\dots,0) : 0\leq k_i\leq v_i-1\mbox{ for all }i\right\}$$
	$$\setminus \bigcup_{i=1}^r\left\{(0,\dots,0,k_iu_i,0,\dots,0,u_{\pi(i)},0,\dots,0) : 0\leq k_i\leq v_i-1\mbox{ for all }i\right\},$$
	where $k_iu_i+1$ and $k_iu_i$ are on position $i$ and $u_{\pi(i)}$ is on position $\pi(i)$.

	\begin{proposition}\label{pr5.8.1}
	The set $A$ has a spectrum (in $G$)
	$$\Lambda_A=\left\{(v_1j_1,\dots,v_rj_r) : 0\leq j_i\leq u_i-1\mbox{ for all } i\right\}.$$
	The sets $B'$ and $B$ have spectrum (in $G$)
	$$\Lambda_B=\left\{(j_1,\dots,j_r) : 0\leq j_i\leq v_i-1\mbox{ for all }i\right\}.$$
	
	Also $\Lambda_A\oplus\Lambda_B=G$. 
	\end{proposition}

	\begin{proof}
	For the set $A$, we check that $\{0,v_i,2v_i, \dots,(u_i-1)v_i\}$ in $\bz_{m_i}$ has spectrum $\{0,1,\dots,u_i-1\}$. Indeed 
	$$\sum_{j=0}^{u_i-1}e^{2\pi i\frac{(k-k')jv_i}{u_iv_i}}=u_i\delta_{kk'}\mbox{ for all }k,k'\in\{0,1,\dots,u_i-1\}.$$
	Then, the result follows from Proposition \ref{pr5.6}. For $B'$, a similar argument can be used. 
	
	For $B$ we will use a lemma:
	\begin{lemma}\label{lem5.10}
	Given $(k_1,\dots, k_r)\in\bz_{m_1}\times\dots\times\bz_{m_r}$, if one of the $k_i$ is a non-zero multiple of $u_i$, then 
	$$\sum_{l_1=0}^{v_1-1}\dots\sum_{l_r=0}^{v_r-1}e^{2\pi i\frac{l_1k_1}{m_1}}\dots e^{2\pi i\frac{l_rk_r}{m_r}}=0.$$
\end{lemma}

\begin{proof}
The sum splits into the product of the sums $\sum_{l_i=0}^{v_i-1}e^{2\pi i\frac{l_ik_i}{u_iv_i}}$, and if $k_i=ku_i$, $0<k<v_i$ then we further get 
$\sum_{l_i=0}^{v_i-1}e^{2\pi i\frac{l_ik}{v_i}}=0$.
\end{proof}

Now, take two distinct points  $b=(b_1,\dots,b_r)\neq b'=(b_1',\dots,b_r')$ in $B$. We want to prove that 

\begin{equation}
\sum_{(l_1,\dots, l_r)\in\Lambda_B}e^{2\pi i\frac{l_1(b_1-b_1')}{m_1}}\dots e^{2\pi i\frac{l_r(b_r-b_r')}{m_r}}=0.
\label{eq5.9.1}
\end{equation}

We denote 
$$B_0=\left\{(u_1k_1,\dots,u_rk_r): 0\leq k_i\leq v_i-1\right\}\setminus$$$$\bigcup_{i=1}^r\left\{(0,\dots,0,k_iu_i,0,\dots,0,u_{\pi(i)},0,\dots,0) : 0\leq k_i\leq v_i-1\mbox{ for all }i\right\},$$
$$B_i=\left\{(0,\dots,0,k_iu_i,0,\dots,0,u_{\pi(i)},0,\dots,0) : 0\leq k_i\leq v_i-1\mbox{ for all }i\right\},$$
$$\tilde B_i=\left\{(0,\dots,0,k_iu_i+1,0,\dots,0,u_{\pi(i)},0,\dots,0) : 0\leq k_i\leq v_i-1\mbox{ for all }i\right\}.$$
So $B=B_0\cup\cup_{i=1}^r \tilde B_i$, disjoint union.

If $b,b'\in B_0$, then the result follows from the fact that $\Lambda_B$ is a spectrum for $B'$, so $B'$ is a spectrum for $\Lambda_B$. If $b\in B_0$ and $b'\in\tilde B_i$, then $b'$ is of the form 
$b'=(0,\dots,0, k_i'u_i+1,0,\dots, 0, u_{\pi(i)},0,\dots, 0)$ and $b$ is of the form $b=(k_1u_1,\dots,k_ru_r)$.

 If one of the $k_1,k_2,\dots,k_{i-1},k_{i+1},\dots, k_{\pi(i)-1},k_{\pi(i)+1},\dots, k_r$ is non-zero, then we apply Lemma \ref{lem5.10} to $b-b'$ and obtain \eqref{eq5.9.1}. 

If all these are zero, then, if $k_{\pi(i)}\neq 1$, again we use Lemma \ref{lem5.10} and obtain \eqref{eq5.9.1}. If $k_{\pi (i)}=1$, then that means that $b\in B_i$, a contradiction. 

If $b'\in\tilde B_i$, and $b$ is in a set $\tilde B_j$ with $j\neq i$, then $b$ is of the form $b=(0,\dots,0, k_ju_j+1,0,\dots, 0, u_{\pi(j)},0,\dots, 0)$. If $\pi(i)= j$ then $\pi(j)\neq i$ (becasue $\pi$ has no cycles of length 2), and $\pi(j)\neq\pi(i)$. Therefore, applying Lemma \ref{lem5.10} to $b-b'$, using the $\pi(j)$ component, we get \eqref{eq5.9.1}.

If $\pi(i)\neq j$, then since $\pi(i)\neq \pi(j)$, we can use the $\pi(i)$ component in Lemma \ref{lem5.10} for $b-b'$ and again obtain \eqref{eq5.9.1}.

Finally, if $b,b'\in \tilde B_i$, then $b-b'\in (B'-B')\setminus\{0\}$ and the result follows, since $\Lambda_B$ is a spectrum for $B'$.

The factorization $\Lambda_A\oplus\Lambda_B=G$ is obvious. 
	\end{proof}
	
	Note that the factorization $\Lambda_A\oplus\Lambda_B=G$ is completely analogous to the factorization $A\oplus B'=G$; the roles of $u_i$ and $v_i$ are interchanged. Therefore, we can perform the same type of operations on $\Lambda_A$ as we did on $B'$ and get new factorizations and spectra. 
	
	So, let us consider a permutation $\sigma$ of $\{1,\dots, r\}$ which has only cycles of length $\geq 3$. Define 
	$$\Lambda_A'=\Lambda_A\cup\left(\bigcup_{i=1}^r([v_ig_i]_{u_i}+v_{\sigma(i)}g_{\sigma(i)})	+g_{\sigma(i)}\right)\setminus\left(\bigcup_{i=1}^r([v_ig_i]_{u_i}+v_{\sigma(i)}g_{\sigma(i)})\right).$$
	
	\begin{proposition}\label{pr5.11}
	$\Lambda_A'$ is a spectrum of $A$ and $\Lambda_A'\oplus\Lambda_B=G$. 
	\end{proposition}

	\begin{proposition}\label{pr5.12} Let $m=m_1\dots m_r$.
	With the isomorphism $\Phi$ in Proposition \ref{pr5.6.1}, 
	the image of the sets $A,B',B,\Lambda_A,\Lambda_A'$ and $\Lambda_B$ in $\bz_{m}$ have the CM-property.
	\end{proposition}
	
	\begin{proof}
	We have 
	$$\Psi(A)=\left\{\sum_{i=1}^rk_i\frac{m}{m_i} : 0\leq k_i\leq u_i-1\right\}.$$
	We use the notation 
	$$Q_n(x)=1+x+\dots+x^{n-1}.$$
	We have $$A(x)=\prod_{i=1}^rQ_{u_i}\left(x^{\frac{m}{m_i}}\right).$$
	Then, with Proposition \ref{pr2.15}(ii) and (vi) (since $u_i$ and $\frac{m}{m_i}$ are mutually prime), we get 
	$$A(x)=\Psi(A)(x)=\prod_{i=1}^r\prod_{d|u_i, d>1}\Phi_{u_i}\left(x^{\frac{m}{m_i}}\right)=\prod_{i=1}^r\prod_{d|u_i, d>1}\prod_{t|\frac{m}{m_i}}\Phi_{dt}(x).$$
	
	If $p^\alpha$, $\alpha>0$ is a prime power such that $\Phi_{p^\alpha}(x)$ divides $A(x)$, then there exist $i\in\{1,\dots,r\}$, $d>1$, $d|u_i$ and $t|\frac{m}{m_i}$ such that $p^{\alpha}=dt$. Since $d$ and $t$ are relatively prime and $d>1$, we must have $p^\alpha=d$. Thus, $S_A$ consists of the prime powers that divide one of the $u_i$. 
	
	Take $p_1^{\alpha_1},\dots, p_n^{\alpha_n}$ in $S_A$. Then $p_1^{\alpha_1}$ divides one of the $u_i$. By relabeling, let us assume $p_1^{\alpha_1},\dots,p_j^{\alpha_j}$ divide $u_i$ and $p_{j+1}^{\alpha_{j+1}},\dots,p_n^{\alpha_n}$ do not, hence they divide some $u_{i_{j+1}},\dots, u_{i_n}$, respectively, different than $u_i$. Then $d=p_1^{\alpha_1}\dots p_j^{\alpha_j}>1$ and $d|u_i$. Also $t=p_{j+1}^{\alpha_{j+1}}\dots p_n^{\alpha_n}$ divides $\frac{m}{m_i}$ therefore $\Phi_{dt}(x)$ divides $A(x)$ and so $A$ has the CM-property.

	For the sets $B'$ and $B$, first note, from Lemma \ref{lemco2.1}, that $S_B=S_{B'}$ consists of prime powers $p^{\alpha}$, $\alpha>0$ which divide $m$ but do not divide any of the $u_i$. 
	
	Take $s_1,\dots,s_n$ prime powers in $S_B$. We will show that, for $s=s_1\dots s_n$, $\Phi_s(x)$ does not divide $A(x)$, so, by Proposition \ref{pr3.2}, it has to divide $B(x)$ and $B'(x)$, which will mean that $B$ and $B'$ satisfy the CM-property. 
	
	Suppose not, so $\Phi_s(x)$ divides $A(x)$, then there exists $i\in\{1,\dots,r\}$, $d>1$, $d|u_i$ and $t|\frac{m}{m_i}$ such that $s=dt$. Each of the $s_k$ divides exactly one of the $m_j$.
	By relabeling, suppose $s_1,\dots,s_j$ are the prime powers that divide $m_i$ and $s_{j+1},\dots,s_n$ do not. Then $d=s_1\dots s_j$. But, since $d|u_i$ it means that $s_1, \dots, s_j$ divide $u_i$, a contradiction. Therefore, none of the $s_k$ divide $m_i$ so all of them divide $\frac{m}{m_i}$. But then, since $s=dt$, it follows that $d=1$, again a contradiction. Thus, $\Phi_s(x)$ does not divide $A(x)$ so it has to divide $B(x)$ and $B'(x)$, which implies that $B$ and $B'$ have the CM-property. 
	
	To show that the $\Lambda$ sets have the CM-property is completely analogous. 
	
	\end{proof}
	\end{example}

	\begin{example}\label{ex5.13}
	This example will show that it is possible for the cyclotomic polynomials to have multiplicity in a tile. They can appear with multiplicity in $A(x)$ and they can appear in both $A(x)$ and $B(x)$. However, from Lemma \ref{lemco2.1}, we know that this is not possible for $\Phi_s(x)$ when $s$ is a prime power. 
	
	Take $m_1=4$, $m_2=9$, $u_1=v_1=2$, $u_2=v_2=3$. Then 
	$$\Psi(A)=\left\{9a+4b : a\in\{0,1\}, b\in\{0,1,2\}\right\},\quad \Psi(B')=\left\{18a+12b : a\in\{0,1\}, b\in\{0,1,2\}\right\}.$$
	
	As in the proof of Proposition \ref{pr5.12}, we have 
	$$A(x)=(1+x^9)(1+x^4+x^8)=\Phi_2(x)\Phi_6(x)\Phi_{18}(x)\Phi_3(x)\Phi_6(x)\Phi_{12}(x).$$
	
	Also, with Lemma \ref{pr2.15},
	$$B'(x)=(1+x^{18})(1+x^{12}+x^{24})=Q_2(x^{18})Q_3(x^{12})=Q_2((x^2)^9)Q_3((x^3)^4)$$$$=\Phi_2(x^2)\Phi_6(x^2)\Phi_{18}(x^2)\Phi_3(x^3)\Phi_6(x^3)\Phi_{12}(x^3)
	=\Phi_4(x)\Phi_{12}(x)\Phi_{36}(x)\Phi_9(x)\Phi_{18}(x)\Phi_{36}(x).$$

	\end{example}

\section{Some general constructions}

\begin{proposition}\label{pr5.0}
Let $A$ be a set of non-negative integers. If $A$ is a complete set of representatives modulo $N$, then $A$ has the CM-property. 
\end{proposition}

\begin{proof}
We have $A\oplus\{0\}=\bz_N$. By Proposition \ref{pr3.2}, there is an integer polynomial such that 
\begin{equation}
A(x)=1+x+\dots+x^{N-1}+(x^N-1)Q(x).
\label{eq5.0.1}
\end{equation}

Since $A$ is a tile, it satisfies (T1), by Theorem \ref{th4.13}. So 
$$N=\#A=\prod_{s\in S_A}\Phi_s(1).$$
Since $\Phi_s(1)=p$ if $s=p^\alpha$ (by Proposition \ref{pr2.15}), it follows that the prime numbers that appear in the prime powers in $S_A$ divide $N$ and also, all prime powers that divide $N$ must appear in $S_A$, otherwise 
$$N=\#A>\prod_{s\in S_A}\Phi_s(1).$$
Thus, $S_A$ consists of all prime powers that divide $N$ so (T2) is also satisfied, because if $s_1,\dots,s_n$ are powers of distinct primes in $S_A$ then $s=s_1\dots s_n$ divides $N$ so $\Phi_s(x)$ divides $A(x)$, by \eqref{eq5.0.1}.
\end{proof}

		\begin{proposition}\label{pr5.1}
	Let $A$ be a finite set of non-negative integers which has a spectrum $\Lambda$ in $\frac{1}{N}\bz$. If $r$ is relatively prime to $N$, then $r\Lambda$ is a spectrum for $A$.
	\end{proposition}
	
	\begin{proof}
	Let $\frac\lambda N\neq \frac{\lambda'} N$ in $\Lambda$ with $\lambda,\lambda'\in\bz$. Then $A(e^{2\pi i\frac{\lambda-\lambda'}N})=0$. Let $\frac{\lambda-\lambda'}N=\frac ks$ with $k,s\in\bz$, $(k,s)=1$. Then $\Phi_s(z)$ divides $A(z)$ since it is the minimal polynomial for $e^{2\pi i\frac ks}$. Since $r$ is relatively prime to $N$, it is also relatively prime to $s$ ($s$ divides $N$). Then $e^{2\pi i\frac{kr}{s}}$ is also a primitive root of order $s$ so $\Phi_s(e^{2\pi i\frac{kr}s})=0$ so 
	$A(e^{2\pi \frac{\lambda-\lambda'}{N}r})=A(e^{2\pi i\frac{kr}{s}})=0$. This shows that $r\Lambda$ is a spectrum for $A$. 
	\end{proof}

For tiles, there is an analogous but more powerful result, due to Tijdeman \cite{MR1345184}. 
	
	\begin{theorem}\label{th3.4} {\bf [Tijdeman's theorem]}
If $A\oplus B=\bz_N$ and $r$ is relatively prime to $\#A$, then $rA\oplus B=\bz_N$.
\end{theorem}

	\begin{question}\label{q5.2}
	Is Proposition \ref{pr5.1} true if $r$ is merely relatively prime to $\#A$? This would be a dual of Tijdeman's theorem. 
	\end{question}

\begin{theorem}\label{th5.2.1}
Let $A$ be a tile in $\bz_N$, $A\oplus B=\bz_N$, let $M$ be some postive integer and, for each $a\in A$, let $A_a$ be a tile in $\bz_M$ with a common tiling set $C$, $A_a\oplus C=\bz_M$. Then the set $\tilde A=\cup_{a\in A}(\{a\}\oplus NA_a)$ is a tile in $\bz_{NM}$ with tiling set $\tilde B=B\oplus NC$. If, in addition, the sets $A$ and $A_a$, $a\in A$ satisfy the CM-property, then the set $\tilde A$ satisfies the CM-property. 
\end{theorem}	
\begin{proof}
Take $x$ in $\bz_{NM}$. It can be written uniquely as $x=k_1+Nk_2$ with $k_1\in\bz_N$ and $k_2\in \bz_M$. Since $A\oplus B=\bz_N$, $k_1$ can be written uniquely as $k_1=a+b$, with $a\in A$ and $b\in B$. Since $A_a\oplus C=\bz_M$, $k_2$ can be written uniquely as $k_2=a'+c$ with $a'\in A_a$ and $c\in C$. Thus $x=(a+Na')+(b+Nc)$. 

Assume now that the sets $A$ and $A_a$, $a\in A$ satisfy the CM-condition. Since $\tilde A$ is a tile, by Theorem \ref{th4.13}, $\tilde A$ satisfies the (T1) property. 
To check the (T2) property, first we have to compute the set $S_{\tilde A}$.

Note first that 
\begin{equation}
\tilde A(x)=\sum_{a\in A}x^a A_a(x^N).
\label{eq5.2.1.1}
\end{equation}

Since $A_a\oplus C=\bz_M$ for all $a$, it follows the the sets $A_a$ have the same cardinality, and, from Lemma \ref{lemco2.1}, that $S_{A_a}$ is the complement of $S_C$ in the set of all prime power factors of $M$. Therefore, the sets $S_{A_a}$ are all equal. By Lemma \ref{lemco1.4}, it follows that the sets $S_{NA_a}$ are all equal, and the sets $NA_a$ satisfy the CM-property.   

We will prove that 
\begin{equation}
S_{\tilde A}=S_A\cup S_{NA_a}\mbox{ disjoint union.}
\label{eq5.2.1.2}
\end{equation}

If $s$ is a prime power in $S_A$, then $s$ divides $N$ (by Lemma \ref{lemco2.1}) and if $\omega=e^{2\pi i/s}$ then $\omega^N=1$. Thus $(NA_a)(\omega)=A_a(\omega^N)=A_a(1)=\#A_a\neq 0$ so $s$ is not in $S_{NA_a}$. So the sets $S_A$ and $S_{NA_a}$ are disjoint. 

With \eqref{eq5.2.1.1}, we have 
$$\tilde A(\omega)=\sum_{a\in A}\omega^aA_a(1)=\#A_a\cdot A(\omega)=0.$$
Thus $s\in S_{\tilde A}$. 

If $s$ is a prime power in $S_{NA_a}$, then $\Phi_s(x)$ divides $(NA_a)(x)=A_a(x^N)$ for all $a\in A$, and by \eqref{eq5.2.1.1}, it follows that $\Phi_s(x)$ divides $\tilde A(x)$, so $s\in S_{\tilde A}$. This proves that $S_A\cup S_{NA_a}\subset S_{\tilde A}$. 

Since the sets $\tilde A$, $A$ and $NA_a$ are tiles, they satisfy the (T1) property. Therefore, we have 
$$\#\tilde A=\prod_{s\in S_{\tilde A}}\Phi_s(1)\geq \prod_{s\in S_A}\Phi_s(1)\prod_{s\in S_{NA_a}}\Phi_s(1)=\#A\cdot\#(NA_a)=\#\tilde A.$$
Thus, we have equality in the inequality, and with Proposition \ref{pr2.15}, since $\Phi_s(1)>1$ for any prime power $s$, it follows that we cannot have more elements in $S_{\tilde A}$, so \eqref{eq5.2.1.2} is satisfied. 

Now take $s_1,\dots,s_n$ powers of distinct primes in $S_{\tilde A}$. If all the $s_i$ are in $S_A$, then since $A$ satisfies the (T2) property, it follows that for $s=s_1\dots s_n$, $\Phi_s(x)$ divides $A(x)$. Also, in this case, $s$ divides $N$, so if $\omega=e^{2\pi i/s}$ then $\omega^N=1$. Using \eqref{eq5.2.1.1}, we obtain that $\tilde A(\omega)=0$ so $\Phi_s(x)$ divide $\tilde A(x)$. 

If all the $s_i$ are in $S_{NA_a}$, then since $NA_a$ satisfies the (T2) property, we get that $\Phi_s(x)$ divides $(NA_a)(x)=A_a(x^N)$, for all $a\in A$, and using \eqref{eq5.2.1.1}, we obtain that $\Phi_s(x)$ divides $\tilde A(x)$. 

Now assume $s_1,\dots,s_j$ are in $S_{A}$ (and hence they divide $N$) and $s_{j+1},\dots, s_{n}$ are in $S_{NA_a}$. Let $s'=s_{j+1}\dots s_n$. We can factor $N$ as $N=N_1N_2N_3$ where $N_1$ contains all the prime factors that appear in $s_1,\dots, s_j$, $N_2$ contains all the prime factors of $N$ that appear in $s'$ and $N_3$ contains all the prime factors of $N$ that do not appear in $s_1,\dots, s_n$, and $N_1,N_2,N_3$ are mutually prime. (The numbers $s_1\dots s_j$ and $s'$ are relatively prime, because the $s_i$ are powers of {\it distinct} primes.) 

Then $\frac{N}{s_1\dots s_n}=\frac{N_1N_2N_3}{s_1\dots s_j s'}$ can be reduced to $k_1\frac{k_2}{s''}N_3$ where $k_1=\frac{N_1}{s_1\dots s_j}$ is an integer which contains only prime factors that appear in $s_1,\dots, s_j$, $\frac{N_2}{s'}=\frac{k_2}{s''}$, with $k_2,s''$ relatively prime integers, and $k_2$ contains only prime factors that appear in $s'$.

Then we also have $\frac{N}{s'}=\frac{N_1N_2N_3}{s'}=N_1\frac{k_2}{s''}N_3$, with $s'$ and $N_1k_2N_3$ relatively prime, so if $\omega'=e^{2\pi i/s}$ then $\omega'^N= e^{2\pi i\frac{N}{s'}}$ is a primitive root of unity of order $s''$. Since $NA_a$ satisfies the (T2) property, we get that $\Phi_{s'}(x)$ divides $(NA_a)(x)$ so $A_a(\omega'^N)=0$ which means that $\Phi_{s''}(x)$ divides $A_a(x)$. 
Then we also have $\frac{N}{s'}=\frac{N_1N_2N_3}{s'}=N_1\frac{k_2}{s''}N_3$, with $s'$ and $N_1k_2N_3$ relatively prime, so if $\omega'=e^{2\pi i/s}$ then $\omega'^N= e^{2\pi i\frac{N}{s'}}$ is a primitive root of unity of order $s''$. Since $NA_a$ satisfies the (T2) property, we get that $\Phi_{s'}(x)$ divides $(NA_a)(x)$ so $A_a(\omega'^N)=0$ which means that $\Phi_{s''}(x)$ divides $A_a(x)$. 

 Since $k_1k_2N_3$ is also relatively prime with $s''$, it follows that, for $\omega=e^{2\pi i/s}$, $\omega^N=e^{2\pi i\frac{k_1k_2N_3}{s''}}$ is also a primitive root of unity of order $s''$. Therefore, $(NA_a)(\omega)=A_a(\omega^N)=0$ which means that $\Phi_s(x)$ divides $(NA_a)(x)=A_a(x^N)$, for all $a\in A$. From \eqref{eq5.2.1.1}, it follows that $\Phi_s(x)$ divides $\tilde A(x)$. 

Thus, $\tilde A(x)$ satisfies the (T2) property.

\end{proof}

\begin{remark}
We see, from Lemma \ref{lem4.5}, that, in the case when $A$ is a tile and $\#A$ has at most two prime factors $p$ and $q$, then reducing modulo $\lcm(S_A)$, either $A$ is contained in $p\bz$ or $q\bz$, or it is of the form given in Theorem \ref{th5.2.1}. We show in the next example that this is not always the case. 
\end{remark}

\begin{example}
Consider Szab{\' o}'s example, with $m_1=2^2$, $m_2=3^2$, $m_3=5^2$, so $G=\bz_4\times\bz_9\times\bz_{25}$ which is isomorphic to $\bz_{900}$ by the isomorphism in Proposition \ref{pr5.6.1}. Let $u_1=v_1=2$, $u_2=v_2=3$, $u_3=v_3=5$ and the permutation $\pi$ of $\{1,2,3\}$, $\pi(1)=2$, $\pi(2)=3$, $\pi(3)=1$. Then 
$$A=\{0,1\}\times \{0,1,2\}\times\{0,1,2,3,4\} \subset \bz_{4}\times\bz_9\times\bz_{25},$$
$$B'=\{0,2\}\times\{0,3,6\}\times\{0,5,10,15,20\}\subset \bz_4\times\bz_9\times\bz_{25}.$$
To construct the set $B$ we replace the subset $B_1=\{0,2\}\times\{3\}\times\{0\}$ of $B'$ with the set $\tilde B_1=B_1+g_1=\{1,3\}\times\{3\}\times\{0\}$, the set $B_2=\{0\}\times\{0,3,6\}\times\{5\}$ with the set $\tilde B_2=B_2+g_2=\{0\}\times\{1,4,7\}\times \{5\}$, and the set $B_3=\{2\}\times\{0\}\times \{0,5,10,15,20\}$ with the set $\tilde B_3=B_3+g_3=\{2\}\times\{0\}\times \{1,6,11,16,21\}$. so
$$B=B'\setminus (B_1\cup B_2\cup B_3)\cup (\tilde B_1\cup\tilde B_2\cup\tilde B_3).$$

The isomorphism in Proposition \ref{pr5.6.1} is given by 
$$\bz_4\times\bz_9\times\bz_{25}\ni (a_1,a_2,a_3)\mapsto a_1\cdot 3^2\cdot 5^2+a_2\cdot 2^2\cdot 5^2+a_3\cdot 2^2\cdot 3^2\in \bz_{900}.$$

We have that $\Psi(B)$, which has the CM-property and tiles $\bz_{900}$, according to Proposition \ref{pr5.12}, does not have the form in Theorem \ref{th5.2.1}. 

Note that if a set $\tilde A$ is of the form in Theorem \ref{th5.2.1}, i.e., 
$$\tilde A=\cup_{a\in a}(\{a\}\oplus NA_a),$$ then, for $k\in\bz_N$, if $k\in A$, then $\#\{a\in \tilde A : a\equiv k (\mod N)\}=\#A_a=\#\tilde A/\#C$, a constant which does not depend on $k$, and which divides $\#\tilde A$, and it is $0$ if $k\not\in A$. We say that $\tilde A$ is {\it equidistributed } $\mod N$. 

We will see that $B$ is not equidistributed $\mod 2$. Indeed, let $B_0:=B'\setminus (B_1\cup B_2\cup B_3)$. 
Then, $$\Psi(B_0)\equiv 0(\mod 2), \Psi(\tilde B_1)\equiv 1 (\mod 2), \Psi(\tilde B_2)\equiv 0(\mod 2), \Psi(\tilde B_3)\equiv 0(\mod 2),$$
$$\Psi(B_0)\equiv 0(\mod 3), \Psi(\tilde B_1)\equiv 0 (\mod 3), \Psi(\tilde B_2)\equiv 1(\mod 3), \Psi(\tilde B_3)\equiv 0(\mod 3),$$
$$\Psi(B_0)\equiv 0(\mod 5), \Psi(\tilde B_1)\equiv 0 (\mod 5), \Psi(\tilde B_2)\equiv 0(\mod 5), \Psi(\tilde B_3)\equiv 1(\mod 5).$$
Thus, $\Psi(B_0)\equiv 0(\mod 30)$, and $\Psi(B_0)$ has 20 elements, a number which does not divide $\#\Psi(B)=30$. This means that $\Psi(B)$ cannot have the form in Theorem \ref{th5.2.1}. 

We have $\Psi(\tilde B_1)=\{75, 525\}$, $\Psi(\tilde B_2)=\{280,580,880\}$, $\Psi(\tilde B_3)=\{486,666,846,126,306\}$, 
$$\Psi(B_0)=\{0,30,60,120,150,210,240,330,360,390,420,510,540,570,600,660,690,720,840,870\}.$$
Thus
$$\Psi(B)=\{0, 30, 60, 75, 120, 126, 150, 210, 240, 280, 306, 330, 360, 390, 420, 486,$$$$ 510, 525, 540, 570, 580, 600, 660, 666, 690, 720, 840, 846, 870, 880\}.$$
\end{example}
	\begin{proposition}\label{pr5.3}
	Let $A$ be a finite set of non-negative integers which has a spectrum $\Lambda_1\subset \frac1N\bz$. Suppose $\{A_a : a \in A\}$ are finite sets of non-negative integers that have a common spectrum $\Lambda_2$. Then the set $\cup_{a\in A}(\{a\}\oplus N A_a)$ is spectral with spectrum $\Lambda_1+\frac1N\Lambda_2$. 
	\end{proposition}
	
	\begin{proof}
	Let $(\lambda_1,\lambda_2)\neq (\lambda_1',\lambda_2')$ in $\Lambda_1\times\Lambda_2$. Since $N(\lambda_1-\lambda_1')\in\bz$, we have
	$$\sum_{a\in A}\sum_{b\in A_a}e^{2\pi i(a+Nb)(\lambda_1-\lambda_1'+\frac1N(\lambda_2-\lambda_2'))}=\sum_{a\in A}e^{2\pi ia(\lambda_1-\lambda_1'+\frac1N(\lambda_2-\lambda_2'))}\sum_{b\in A_a}e^{2\pi ib(\lambda_2-\lambda_2')}.$$
	Since $\Lambda_2$ is a spectrum for all sets $A_a$, we get that, if $\lambda_2\neq \lambda_2'$, the last sum is 0, for all $a\in A$. 
	
	If $\lambda_2=\lambda_2'$, then the above sum becomes
	$$\sum_{a\in A}e^{2\pi i a(\lambda_1-\lambda_1')}\#A_a=\#\Lambda_2\sum_{a\in A}e^{2\pi ia(\lambda_1-\lambda_1')}$$
	and since $\Lambda_1$ is a spectrum for $A$, this sum is 0 if $\lambda_1\neq\lambda_1'$. 
	\end{proof}

	\begin{corollary}\label{cor5.3.1}
	Let $A$ be a spectral set in $\bz_N$ with spectrum $\Lambda_1$. Suppose $\{A_a : a\in A\}$ are subsets of $\bz_M$ that have a common spectrum $\Lambda_2$ in $\bz_M$. Then the set $\cup_{a\in A}(\{a\}\oplus NA_a)$ is spectral in $\bz_{NM}$ with spectrum $M\Lambda_1+\Lambda_2$. 
	\end{corollary}

	Next, we present a result which generalizes and refines a result from Terence Tao's blog, due to I. \L aba.
		
\begin{theorem}\label{thtaolaba}
Suppose $A\oplus B=\bz_N$ and $\#A$ and $\#B$ are relatively prime. Then $A$ consists of a single representative from each class modulo $N/\#B$, and $B$ consists of a single representative from each class modulo $N/\#A$. In other words, $A$ is a complete set of representatives modulo $N/\#B$. In particular, $A$ and $B$ have the CM-property and are spectral. 
\end{theorem}	
	
	\begin{proof}
	Let $r=\#B$. Since $\#A$ and $r$ are relatively prime, by Theorem \ref{th3.4}, $rA\oplus B=\bz_N$. So $B\oplus (rA\oplus N\bz)=\bz$. Then, by Lemma \ref{lemco2.5}, if $B'=\{b\in B : b\equiv 0(\mod r)\}$, we have $B'\oplus (A\oplus N/r\bz)=\bz$ and $\#B'=\#B/r=1$. Thus $A\oplus N/r\bz=\bz$, which means that $A$ consists of a single representative from each class modulo $N/r$. By symmetry, the statement is also true for $B$. 
	\end{proof}
	
	\begin{remark}\label{remtl1}
	If $A$ is a complete set of representatives modulo $k$ and $N$ is a multiple of $k$, $N=kl$, then $A$ always tiles $\bz_N$ with $B=\{ki : i=0,\dots,l-1\}$ which is a subgroup of $\bz_N$. Note however that the tile $B$ does not have to be a subgroup of $\bz_N$, even when $\#A=k$ and $\#B=l$ are mutually prime. For example, 
	$\{0,5,6,11\}\oplus\{0,2,10\}=\bz_{12}$ and neither of the two sets is a subgroup of $\bz_{12}$.
	\end{remark}

	\section{Appendix}
	
		\begin{proposition}\label{pr5.6}
	If $A_1,\dots, A_n$ are finite spectral sets in $G_1,\dots, G_n$ respectively, with corresponding spectra $\Lambda_1,\dots,\Lambda_n$, then $A_1\times\dots\times A_n$ is spectral in $G_1\times\dots\times G_n$ with spectrum $\Lambda_1\times\dots\times\Lambda_n$.
	\end{proposition}
	
	\begin{proof}
	Let $(\varphi_1,\dots,\varphi_n),(\varphi_1',\dots,\varphi_n')\in \Lambda_1\times\dots\times\Lambda_n$. Then 
	$$\sum_{(a_1,\dots,a_n)\in A_1\times\dots\times A_n}\varphi_1(a_1)\dots\varphi_n(a_n)\cj\varphi_1'(a_1)\dots\cj\varphi_n'(a_n)$$$$=
	\left(\sum_{a_1\in A_1}\varphi_1(a_1)\cj\varphi_1'(a_1)\right)\dots\left(\sum_{a_n\in A_n}\varphi_n(a_n)\cj\varphi_n'(a_n)\right)=\#A_1\delta_{\varphi_1\varphi_1'}\dots\#A_n\delta_{\varphi_n\varphi_n'}.$$
	\end{proof}

	\begin{lemma}\label{lema1}
	Let $A$ and $B$ be subsets of $\bz_N$ such that $\#A\cdot\#B=N$. Then $A\oplus B=\bz_N$ if and only if $(A-A)\cap(B-B)=\{0\}$. 
	\end{lemma}
	
	\begin{proof}
	The direct implication is clear. For the converse, define the map $A\times B \rightarrow\bz_N$, $(a,b)\mapsto a+b$. The condition implies that the map is injective. Since the two sets have the same cardinality $N$, the map is also surjective, so $A\oplus B=\bz_N$.
	\end{proof}
	
		\begin{lemma}\label{lemco2.1}\cite{CoMe99}
	Let $A(x)$ and $B(x)$ be polynomials with coefficients 0 and 1, $n=A(1)B(1)$, and $R$ the set of prime power factors of $N$. If $\Phi_t(x)$ divides $A(x)$ or $B(x)$ for every factor $t>1$ of $N$, then 
	\begin{enumerate}
		\item $A(1)=\prod_{s\in S_A}\Phi_s(1)$ and $B(1)=\prod_{s\in S_B}\Phi_s(1)$. 
		\item $S_A$ and $S_B$ are disjoint sets whose union is $R$. 
	\end{enumerate}
	\end{lemma}

	\begin{proposition}\label{pr2.15}\cite{CoMe99}
Let $p$ be a prime.
\begin{enumerate}
	\item A polynomial $P(x)\in\bz[x]$ is divisible by $\Phi_s(x)$ if and only if $P(\omega)=0$ for a primitive $s$-th root of unity $\omega$.
	\item $1+x+\dots+x^{s-1}=\prod_{t>1, t | s}\Phi_t(x)$.
	\item $\Phi_p(x)=1+x+\dots+x^{p-1}$ and $\Phi_{p^{\alpha+1}}(x)=\Phi_p(x^{p^\alpha})$.
	\item 
	$$\Phi_s(1)=\left\{\begin{array}{cc}
	0&\mbox{ if }s=1\\
	q&\mbox{ if $s$ is a power of a prime $q$}\\
	1&\mbox{ otherwise.}
	\end{array}\right.$$
	\item 
	$$\Phi_s(x^p)=\left\{\begin{array}{cc}
	\Phi_{ps}(x)&\mbox{ if $p$ is a factor of $s$}\\
	\Phi_s(x)\Phi_{ps}(x)&\mbox{ if $p$ is not a factor $s$}.
	\end{array}\right.$$
	\item If $s$ and $t$ are relatively prime, then $\Phi_s(x^t)=\prod_{r|t}\Phi_{rs}(x)$.
	
\end{enumerate}

\end{proposition}

	\begin{lemma}\cite{CoMe99}\label{lemco1.4}
	Let $k>1$ and let $A=k\cj A$ be a finite set of non-negative integers. 
	\begin{enumerate}
		\item $A$ tiles the integers if and only if $\cj A$ tiles the integers.
		\item If $p$ is prime, then $S_{p\cj A}=\{p^{\alpha+1} : p^\alpha\in S_{\cj A}\}\cup \{q^\beta\in S_{\cj A} : q\mbox { prime },q\neq p\}.$
		\item $A(x)$ satisfies (T1) if and only if $\cj A(x)$ satisfies (T1).
		\item $A(x)$ satisfies (T2) if and only if $\cj A(x)$ satisfies (T2). 
	\end{enumerate}
	\end{lemma}

	\begin{lemma}\label{lemco2.5}\cite{CoMe99}
	Suppose $A\oplus C=\bz$, where $A$ is a finite set of non-negative integers, $k>1$ and $C\subset k\bz$. For $i=0,1,\dots,k-1$, let $A_i=\{a\in A : A\equiv i(\mod k)\}$, $a_i=\min(A_i)$, and $\cj A_i=\{a-a_i : a\in A_i\}/k$. Then 
	\begin{enumerate}
		\item $A(x)=x^{a_0}\cj A_0(x^k)+x^{a_1}\cj A_1(x^k)+\dots+x^{a_{k-1}}\cj A_{k-1}(x^k)$. 
		\item Every $\cj A_i\oplus C/k=\bz$. 
		\item The elements of $A$ are equally distributed modulo $k$, i.e., every $\#\cj A_i=(\#A)/k$. 
		\item $S_{\cj A_0}=S_{\cj A_1}=\dots=S_{\cj A_{k-1}}$. 
		\item When $k$ is prime $S_A=\{k\}\cup S_{k\cj A_0}$ and if every $\cj A_i(x)$ satisfies (T2), then $A(x)$ satisfies (T2). 
	\end{enumerate}
	\end{lemma}

 \noindent {\it Acknowledgments} This work was partially supported by a grant from the Simons Foundation (\#228539 to Dorin Dutkay). Most of the paper was written while Dorin Dutkay was visiting the Institute of Mathematics of the Romanian Academy, with the support of the Bitdefender Invited Professorship. We would like to thank Professors Lucian Beznea, {\c S}erban Stratila and Dan Timotin for their kind hospitality.

\bibliographystyle{alpha}	
\bibliography{eframes}
\end{document}